\documentclass[10pt,a4paper]{article}

\usepackage[utf8]{inputenc}
\usepackage[T2A]{fontenc}
\usepackage[margin=1cm]{geometry}

\usepackage{intcalc, calc}
\usepackage[nomessages]{fp}
\usepackage{graphicx}
\usepackage{wrapfig}
\usepackage{float}

\usepackage{tikz}
\usepackage{color}
\usepackage{framed}
\usepackage{ragged2e}
\usepackage{amsmath, amssymb, amsfonts, amsthm}
\usepackage{wasysym}
\usepackage{soul, soulutf8}
\usepackage{calc}
\usepackage{mathtools,mathcomp}
\usepackage{pifont}
\usepackage{enumitem}

\usepackage{ifthen}

\DeclareMathOperator{\dist}{\operatorname{dist}}

\def\divdots{\rlap{\raisebox{-1pt}{.}}{\rlap{\raisebox{2pt}{.}}\raisebox{5pt}{.}}}

\def\ndivby{\mathrel{
    \divdots
    \kern-0.35em\raise0.22ex\hbox{/}
}}

\renewcommand{\leq}{\leqslant}
\renewcommand{\geq}{\geqslant}

\def\fs{\kern 0.5em}
\newcounter{prcnt}
\newcounter{pucnt}
\newcommand{\prmain}[1]{ 
    \medskip%
    \setcounter{pucnt}{0}%
    \stepcounter{prcnt}%
    \noindent\textbf{%
    \theprcnt%
    \ifthenelse{\equal{#1}{1}}{*}{}%
    .}\fs%
}

\newcommand{\pumain}[1]{
    \stepcounter{pucnt}{%
    \noindent\bf(\alph{pucnt}%
    \ifthenelse{\equal{#1}{1}}{*}{}%
    )}\fs%
}

\usepackage{longtable,wrapfig}
\usepackage{setspace}
\usepackage{listings}
\usepackage{pgf,tikz,pgfplots}
\usepackage{mathrsfs}
\usetikzlibrary{arrows}
\usepackage[ruled,vlined]{algorithm2e}

\newtheorem{theorem}{Theorem}

\newtheorem{lemma}{Lemma}

\newtheorem*{theorem*}{Теорема}
\newtheorem{corollary}{Corollary}
\theoremstyle{definition}

\pgfplotsset{compat=1.17}

\title{On the chromatic number of 2-dimensional spheres}
\author{Danila Cherkashin$^\mathrm{a,b,c}$,~
 Vsevolod Voronov$^{\mathrm{d}}$\\
{\tiny ~a. Chebyshev Laboratory, St. Petersburg State University, 14th Line V.O., 29, Saint Petersburg 199178 Russia}\\
{\tiny ~b. Moscow Institute of Physics and Technology, Laboratory of Combinatorial and Geometric Structures}\\ 
{\tiny ~c. St. Petersburg Department of Steklov Mathematical Institute of Russian Academy of Sciences
27 Fontanka, St. Petersburg, Russia }\\
{\tiny ~d. Caucasus Mathematical Center of Adyghe State University }}
\date{}

\begin{document}

\maketitle

\begin{abstract}
In 1976 Simmons conjectured that every coloring of a 2-dimensional sphere of radius strictly greater than $1/2$ in three colors 
has a pair of monochromatic points at the distance 1 apart. We prove this conjecture.
\end{abstract}

\section{Introduction}

A \textit{coloring} of a given set $M$ is a map from $M$ to the set of colors.
A coloring of a subset $M$ of a metric space is \textit{proper} if no pair of monochromatic points lie at distance 1 apart.
The minimum number of colors that admits a proper coloring of $M$ is called \textit{the chromatic number} of $M$; we denote it by $\chi (M)$. 
In the case of $M\subset\mathbb{R}^n$, the distance typically comes from the induced Euclidean metric on $M$.

A slightly different point of view is to consider a \textit{unit distance graph} $G(M)$: the points of $M$ are the vertices of $G(M)$ and edges connect points at unit distance apart. 
By definition, $\chi(M) = \chi(G(M))$.
The de~Bruijn--Erd{\H o}s theorem states that if $\chi(M)$ is finite then there is a finite subgraph $H$ of $G(M)$ such that $\chi(H) = \chi(G(M))$.

Denote by $S^2(r)$ the two-dimensional sphere of radius $r$ in $\mathbb{R}^3$ centered at the origin.
Let $\chi(S^2(r))$ be the chromatic number of $S^2(r)$ with respect to the Euclidean metric. 
Obviously if $r < 1/2$ and $r = 1/2$ then the chromatic number is equal to 1 and 2, respectively.
Note that for any $r>\frac{1}{2}$ there is $r_1<r$ such that $S^1(r_1)$ contains an odd cycle. Since $S^1(r_1) \subset S^2(r)$, we obtain that $\chi(S^2(r)) \geq 3$. 
G.~Simmons~\cite{simmons1976chromatic} proved that 
\[
    \chi(S^2(r)) \geq 4 \quad \mbox{for} \quad r \geq \frac{\sqrt{3}}{3}.
\]
In the proof, Simmons constructs certain subgraphs of $G(S^2(r))$ that contain triangles. Obviously, for smaller values of the radius $G(S^2(r))$ is triangle-free, and so other ideas are needed.

Then L. Lov{\'a}sz~\cite{Lovasz} generalized the odd cycle construction to an arbitrary dimension, showing that for every $n \geq 3$ there exists a family of \emph{strongly self-dual polytopes} inscribed in $S^{n-1}(r)$ whose graphs of diameters have chromatic number $n+1$ and that $r$ can be arbitrarily close to $\frac{1}{2}$.
In our notation this result can be formulated as follows:

\begin{theorem}[Lov{\'a}sz,~\cite{Lovasz}]
For every $n \geq 2$ there exists a monotonically decreasing sequence $r_k^{(n)}, k = 1,2, \dots$, such that 
\[
\lim_{k \to \infty} r_k^{(n)}=\frac{1}{2} \quad \mbox{and} \quad \chi\left(S^{n-1}\left(r_k^{(n)}\right)\right) \geq n+1.
\]
\end{theorem}

Since $S^{n-1}(r_1) \subset S^{n}(r)$ for $r_1 \leq r$, we get the following inequality.

\begin{corollary}
\[\chi(S^{n-1}(r)) \geq n \quad \mbox{for} \quad r > \frac{1}{2}.\]
\end{corollary}

Some sources state that the chromatic number of a two-dimensional sphere $S^2(r)$ is known only for $r\leq \frac{1}{2}$ and for $r = \frac{\sqrt{2}}{2}$ \cite{Jensen, Malen}. But it should be clarified that the equality $\chi(S^2(r))=n+1=4$ is true for $r \in  \{r^{(3)}_k\} \cap \left(\frac{1}{2}, \frac{\sqrt{3-\sqrt{3}}}{2}\right]$. Explicit formulas for algebraic numbers $r^{(3)}_k$, if such exist, seem to be too complicated, but it is not difficult to compute $r^{(3)}_k$ for a given $k$ with an arbitrary precision by approximately solving a certain optimization problem. For example, the first non-trivial construction in the case of a two-dimensional sphere corresponds to a unit distance embedding of the Gr{\" o}tzsch graph at $r = 0.54003829...$



It is worth noting that chromatic numbers in high dimensions were studied using algebraic, topological and combinatorial methods.
A.M.~Raigorodskii~\cite{Rai12} showed that for every fixed $r>1/2$ the chromatic number of an $n$-dimensional sphere grows exponentially with $n$.  O.~Kostina~\cite{kostina} refined asymptotic lower bounds.  R.~Prosanov~\cite{Prosanov} gave a new asymptotic upper bound.
The paper of A.~Kupavskii~\cite{kupavskii} contains several results on the number of different colors on a sphere of given radius in every proper coloring of $\mathbb{R}^n$.

A lot of results on colorings of 2-dimensional spheres were obtained by Simmons~\cite{simmons1976chromatic}. 
Recent discovery of a 5-chromatic unit distance subgraph of the Euclidean plane~\cite{deGrey} spurred 
interest to the topic and in particular to the chromatic number of a 2-dimensional sphere. 

Among the other results, in~\cite{sphere5chr} the authors constructed several 5-chromatic subgraphs of 2-dimensional spheres, which lead to the bounds
\[
\chi(S^2(r_1)) \geq 5 \quad \mbox{where} \quad  r_1 = \cos \frac{3\pi}{10} =  \frac{\sqrt{5-\sqrt{5}}}{2\sqrt{2}} = 0.58778\dots ;\
\]
\[
\chi(S^2(r_2)) \geq 5 \quad \mbox{where} \quad r_2 = \cos \frac{\pi}{10}=\frac{\sqrt{5+\sqrt{5}}}{2\sqrt{2}} = 0.95105\dots.
\]
The paper~\cite{Sirgedas} contains a family of proper colorings of $S^2(r)$ spheres in 7 colors, provided $r$ is large enough.

The following statement was formulated by Simmons as a conjecture~\cite{simmons1976chromatic}. The proof of Simmons' conjecture is the main result of the present paper.
\begin{theorem}
For every $r > \frac{1}{2}$ we have
\[
\chi(S^2(r))\geq 4.
\]
\end{theorem}
We note that for $\frac{1}{2} < r \leq \frac{\sqrt{3-\sqrt{3}}}{2}=0.563\dots$ a proper 4-coloring of $S^2(r)$ can be obtained from a partition of the sphere into four equal spherical triangles~\cite{simmons1976chromatic}. It implies the following corollary.
\begin{corollary}
$\chi(S^2(r)) = 4$ for $\frac{1}{2} < r \leq\frac{\sqrt{3-\sqrt{3}}}{2} = 0.563\dots$.
\end{corollary}

\paragraph{Structure of the paper.} Section 2 contains the proof of Theorem 2. 
In Section 3 we summarize the results and discuss some further questions.

\section{Proof of Theorem 2}

Recall that for $r \geq \frac{\sqrt{3}}{3}$ the statement was proved in~\cite{simmons1976chromatic}. 

Here is the sketch of the proof. Fix $r\in \left(\frac{1}{2} , \frac{\sqrt{3}}{3}\right)$. 
The proof consists of two steps. Suppose that there is a proper 3-coloring of the sphere $S^2(r)$.
In the first step we use the Borsuk--Ulam theorem to show that every color is dense in the sphere.
Consider a graph $G_k$ with vertices $x_1,\dots x_{2k+1}$, $y_1,\dots, y_{2k+1}$ and edges $\{(y_i,y_{i+1}), (x_i,y_i):1 \leq i \leq 2k+1\}$ 
(where indices are modulo $2k+1$). 
We provide an explicit representation of $G_k$ as a unit distance subgraph of the sphere. 
The second step is to show that this embedding is stable under small perturbations of $x_i$.
Then one can move every $x_i$ at a red point, which forces the odd cycle on vertices $y_i$ to be colored in the remaining two colors. 
The contradiction proves the theorem.

Note that the idea of attaching an odd cycle to a finite set $A$ in order to exclude the possibility of $A$ to be monochromatic was used in a series of papers devoted to the existence of planar unit distance graphs with chromatic number 4 and arbitrarily large girth~\cite{odonnell,Soifer, wormald}. 
The key twist in step 2 is to find the required embedding of $G_k$ implicitly, i.e. the corresponding $A$ is not a constructive set.
Similar ideas were used by the authors in~\cite{kanel18}.

\subsection{Step 1. Each color is a dense set}

All the distances are considered in the metrics induced from Euclidean space $\mathbb{R}^3$, the distance between $x$ and $y$ is denoted by $\|x-y\|$.

Fix $r\in \left(\frac{1}{2} , \frac{\sqrt{3}}{3}\right)$ and consider $S^2(r)$.
Suppose that there is a proper coloring of $S^2(r)$ in three colors. 
Consider the unit distance graph $G = G(S^2(r))$.
Then neighborhood of a vertex in $G$ forms a circle of diameter $d = \frac{\sqrt{4r^2-1}}{r}$ and radius $\rho = \sqrt{4r^2-1}$ in the induced metric, centered at the opposite point of the sphere.
Vice versa, any circle of such radius is a graph-neighborhood of some vertex, and hence contains points of at most two colors.
We need the following technical statement.
\begin{lemma}
Let $D \subset S^2(r) \times S^2(r)$ be a set of pairs $(x,y)$ such that $0 < \| x - y\| < d$. Then
\begin{itemize}
    \item for every $(x,y) \in D$ there are two circles of radius $\rho$ containing $x$ and $y$. One may denote their centers by $c_r$ and $c_l$ in such a way that the triple of radius-vectors $(x,y,c_r)$ is right-handed and the triple $(x,y,c_l)$ is left-handed.
    \item The functions $c_r(x,y)$ and $c_l(x,y)$ from $D$ to $S^2(r)$ are continuous.
\end{itemize}
\label{circ_continuous}
\end{lemma}

In what follows, we will call a circle passing through the points $x$, $y$ with center $c$ \emph{right-handed} if the triple $(x,y,c)$ is right-handed, and \emph{left-handed} otherwise.

Let $C_{red}$, $C_{blue}$, $C_{green}$ be the sets of red, blue and green points, respectively. 
A \textit{chromaticity} of a point $x$ is the number of sets $\overline{C_{red}}$, $\overline{C_{blue}}$, $\overline{C_{green}}$ containing $x$ (as usual, $\overline{T}$ stands for the closure of a set $T$).
A set $T \subset S^2(r)$ is called \textit{dense} if $\overline{T} = S^2(r)$. 
Let $B_\rho(x)$ denote the set of points $y \in S^2(r)$ such that $\|x - y\| < \rho$, i.e. an open ball of radius $\rho$ and diameter $d$.

\begin{lemma}
If some open ball of diameter $d$ contains points of all three colors then each of $C_{red}$, $C_{blue}$, $C_{green}$ is dense in the sphere.
\label{dense}
\end{lemma}

\begin{proof}
Consider points $x \in C_{red}$, $y \in C_{blue}$ and $z \in C_{green}$ inside a ball $K_0$ of diameter $d$. 
Then one can continuously move $K_0$ to a ball $K$ containing two points (say, $x$ and $y$) on the boundary; 
at the first such moment the point $z$ lies inside $K$. 
The circle $\partial K$ contains blue and red points and so it is colored in blue and red only. Hence, it contains a point $u$ lying in the closures of $C_{red}$ and $C_{blue}$; without loss of generality, assume that point $u$ is red. 
A red-green circle (right-handed, see Lemma~\ref{circ_continuous}) of diameter $d$ containing $z$ and $u$  and a blue-green circle (left-handed) with the diameter $d$ containing $z$ and blue point $u'$ in a small neighborhood of $u$ intersect in a green point $v$.
Note that if $u=u'$ then $v=u=u'$. Hence, due to the continuity of circles in Lemma~\ref{circ_continuous}, $v$ may be arbitrarily close to $u$ with a proper choice of $u'$ (see Fig.~\ref{chrom3}). It implies that the chromaticity of $u$ is three.

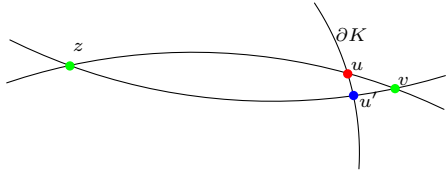
\begin{figure}[H]
    \centering
    \begin{tikzpicture}[line cap=round,line join=round,>=triangle 45,x=1.0cm,y=1.0cm, scale=0.70]
\draw [shift={(10.08,4.04)}] plot[domain=4.25:5.01,variable=\t]({1*11.07*cos(\t r)+0*11.07*sin(\t r)},{0*11.07*cos(\t r)+1*11.07*sin(\t r)});
\draw [shift={(6.54,-8)}] plot[domain=-0.06:0.58,variable=\t]({1*5.17*cos(\t r)+0*5.17*sin(\t r)},{0*5.17*cos(\t r)+1*5.17*sin(\t r)});
\draw [shift={(8.58,-16.93)}] plot[domain=1.12:1.9,variable=\t]({1*10.83*cos(\t r)+0*10.83*sin(\t r)},{0*10.83*cos(\t r)+1*10.83*sin(\t r)});
\begin{scriptsize}
\fill [color=green] (6.27,-6.36) circle (2.5pt);
\draw[color=black] (6.43,-6.0) node {$z$};
\fill [color=green] (12.38,-6.79) circle (2.5pt);
\draw[color=black] (12.54,-6.65) node {$v$};
\draw[color=black] (11.56,-5.76) node {$\partial K$};
\fill [color=red] (11.49,-6.51) circle (2.5pt);
\draw[color=black] (11.67,-6.36) node {$u$};
\fill [color=blue] (11.6,-6.92) circle (2.5pt);
\draw[color=black] (11.89,-7.0) node {$u'$};
\end{scriptsize}
\end{tikzpicture}
    \caption{Finding a point with chromaticity 3 in Lemma~\ref{dense}}
    \label{chrom3}
\end{figure}

Since $u$ has chromaticity 3, a small neighborhood of $u$ contains a point $a \neq u$ with the chromaticity at least 2.
Suppose that $a$ has chromaticity 2 (say, $a$ does not lie in $\overline{C_{green}}$) and $\|a-u\|<d$. 
Consider a green point $b$ in a small neighborhood of $u$. Consider a red point $e$ and a blue point $f$ in a small neighborhood of $a$.
Then the right-handed circle containing $b$ and $e$ is red-green and the left-handed circle containing $b$ and $f$ is blue-green, so they intersect in a green point $g$. Since the neighborhoods can be chosen arbitrarily small, $g$ can be arbitrarily close to $a$. Hence $a$ has chromaticity 3, a contradiction. 

Thus we have shown that if a point with the chromaticity 3 and a point with the chromaticity at least 2 lie at a distance smaller than $d$, then they both have chromaticity 3.

\begin{figure}[H]
    \centering
    \begin{tikzpicture}[line cap=round,line join=round,>=triangle 45,x=1cm,y=1cm, scale=0.6]
\draw [line width=0.4pt] (-3.5345454545454498,-1.458181818181819) circle (5cm);
\draw [line width=0.4pt,dotted] (-5.740650126700984,-5.945174371718489) circle (1cm);
\draw [line width=0.4pt,dotted] (-1.567168276054479,-6.054857468861656) circle (1cm);
\draw [line width=0.4pt,dotted] (-1.7432879583165823,3.209943778050987) circle (1cm);
\draw [line width=0.4pt,dash pattern=on 1pt off 1pt] (-3.8499490326019727,-1.8997614922280406) circle (5.004384588957655cm);
\begin{scriptsize}
\draw [fill=black] (-3.5345454545454498,-1.458181818181819) circle (2.5pt);
\draw[color=black] (-3.172727272727266,-0.985) node {$c$};
\draw[color=black] (2,-0.985) node {$L$};
\draw [fill=red] (-5.740650126700984,-5.945174371718489) circle (2.5pt);
\draw[color=red] (-5.328727272727267,-5.407) node {$x_1$};
\draw [fill=blue] (-1.567168276054479,-6.054857468861656) circle (2.5pt);
\draw[color=blue] (-1.5447272727272658,-5.429) node {$x_2$};
\draw [fill=red] (-1.7432879583165823,3.209943778050987) circle (2.5pt);
\draw[color=red] (-1.390727272727266,3.679) node {$y_1$};
\draw [fill=blue] (-6.243636363636359,-6.294545454545455) circle (2.5pt);
\draw[color=blue] (-6.472727272727267,-6.661) node {$u_1$};
\draw [fill=green] (-5.970909090909086,-5.44) circle (2.5pt);
\draw [fill=green] (-1.8254545454545408,-6.4763636363636365) circle (2.5pt);
\draw[color=green] (-1.8307272727272659,-6.881) node {$v_1$};
\draw [fill=red] (-1.1487272727272657,-6.518) circle (2.5pt);
\draw [fill=blue] (-2.1709090909090865,2.8145454545454536) circle (2.5pt);
\end{scriptsize}
\end{tikzpicture}
    \caption{Propagation of 3-chromaticity along a circle in Lemma~\ref{dense}}
    \label{chrom_circ}
\end{figure}
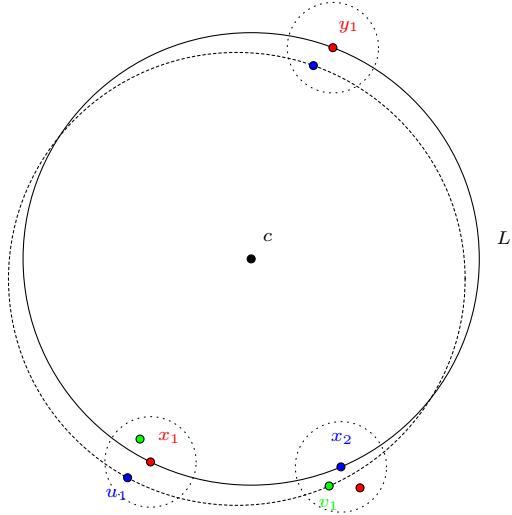

Now let $x_1$ and $x_2$ be points of chromaticity 3 such that $\|x_1-x_2\| < d$. 
We claim that any point on a circle $L$ of diameter $d$ containing $x_1$ and $x_2$ has chromaticity three. 
By the previous argument it is enough to show that the chromaticity is at least 2. 
Without loss of generality, a triple $(x_1,x_2,c)$ is left-handed, where $c$ is the center of $L$ on the sphere.
Arguing indirectly, assume that a point $y_1 \in L$ has a small red neighborhood $U_{y_1}$. Choose a blue point $u_1$ in a small neighborhood of $x_1$ and a green point $v_1$ in a small neighborhood of $x_2$ (see Fig. \ref{chrom_circ}). 
By Lemma~\ref{circ_continuous} the left-handed circle of diameter $d$ passing through blue point $u_1$, green point $v_1$ is close to $L$ so it intersects red set $U_{y_1}$; this contradiction shows that every point on $L$ has chromaticity 3.

Let $q$ be an arbitrary point of $S^2(r)$. Consider a path $q_0, q_1 \dots q_t = q$ such that $q_0 \in L$ and $\|q_{i+1} - q_i \| < \rho$ for $0 \leq i \leq t-1$.
A circle $L_1$ of diameter $d$ that passes through $q_{1}$ and $q_{0}$ intersects $L$ in two points, so by the previous argument every point (in particular, $q_1$) of $L_1$ has chromaticity 3. By induction, a circle $L_{i+1}$ of diameter $d$ that passes through $q_{i+1}$ and $q_{i}$ intersects $L_{i}$ in two points, so every point in $L_{i+1}$ (in particular $q_{i+1}$) has chromaticity 3. So $q = q_t$ also has chromaticity 3. Since $q \in S^2(r)$ was arbitrary, every point of  $S^2(r)$ has chromaticity 3.


\end{proof}


Suppose that the condition of Lemma~\ref{dense} does not hold, i.e. 
\begin{equation}\tag{$\star$}
\mbox{\emph{every open ball of diameter $d$ contains points of at most two colors}}.
\end{equation}
Consider a continuous function  
\[
f: S^2(r) \to \mathbb{R}^2, \quad f(x) = (\dist (x,\overline{C_{red}}), \dist (x,\overline{C_{blue}})),
\]
where $\dist(\cdot)$ stands for the distance between a point and a set in $\mathbb{R}^3$.
By the Borsuk--Ulam theorem there exists $x^* \in S^2(r)$ such that $f(x^*)=f(-x^*)$. We have to deal with three cases.

\begin{figure}[H]
    \centering
\begin{tikzpicture}[line cap=round,line join=round,>=triangle 45,x=1.0cm,y=1.0cm, scale=0.6]
\clip(-3,-8) rectangle (13,4);
\draw(4.92,-1.94) circle (7.4cm);
\draw [rotate around={-0.33:(4.94,-1.94)},dotted] (4.94,-1.94) ellipse (7.39cm and 1.46cm);
\draw [rotate around={-59.24:(10.3,1.19)},dash pattern=on 2pt off 2pt] (10.3,1.19) ellipse (2.56cm and 0.92cm);
\draw [rotate around={-60.56:(-0.47,-5.05)},dash pattern=on 2pt off 2pt] (-0.47,-5.05) ellipse (2.56cm and 0.92cm);
\draw [shift={(13.03,4.09)},dash pattern=on 2pt off 2pt]  plot[domain=3.64:4.1,variable=\t]({1*4.82*cos(\t r)+0*4.82*sin(\t r)},{0*4.82*cos(\t r)+1*4.82*sin(\t r)});
\draw [dotted] (-0.78,-5.24)-- (-1.74,-2.83);
\draw [dotted] (-1.74,-2.83)-- (0.81,-7.26);
\begin{scriptsize}
\fill [color=black] (4.92,-1.94) circle (1.5pt);
\fill [color=black] (10.55,1.32) circle (1.5pt);
\draw[color=black] (10.9,1.22) node {$-x^*$};
\fill [color=blue] (-0.78,-5.24) circle (1.5pt);
\draw[color=blue] (-1.08,-5.14) node {$x^*$};
\fill [color=red] (-0.36,-5.58) circle (1.5pt);
\draw[color=red] (-0.38,-5.8) node {$z$};
\fill [color=black] (10.2,1.7) circle (1.5pt);
\draw[color=black] (10.1,2.04) node {$-z$};
\fill [color=green] (9.39,0.93) circle (1.5pt);
\draw[color=green] (9.06,0.6) node {$y_1$};
\draw[color=black] (-1.56,-4.18) node {$\rho$};
\draw[color=black] (-0.02,-5.32) node {$d$};
\fill [color=blue] (10.4,0.92) circle (1.5pt);
\fill [color=red] (10.12,1.22) circle (1.5pt);
\draw [dash pattern=on 5pt off 5pt] (-0.78,-5.24)-- (10.87,-0.85);
\draw[color=black] (5.08,-1.68) node {$0$};
\draw[color=black] (6.66,-2.68) node {$1$};
\draw[color=black] (0.72,2.0) node {$S^2(r)$};
\draw[color=black] (1.78,-6.86) node {$B_\rho(x^*)$};
\draw[color=black] (7.8,2.76) node {$B_\rho(-x^*)$};
\end{scriptsize}
\end{tikzpicture}
    \caption{Case 1}
    \label{case1}
\end{figure}

\paragraph{Case 1: $f(x^*) = (0,0)$.} Without loss of generality, the point $x^*$ is blue. 
One may pick a red point $z$, which is arbitrarily close to $x^*$. 
If $\|x-z\| < \rho$, then the intersection of circles of unit Euclidean radius with centers $x^*$ and $z$ consists of two green points $y_1,y_2$ belonging to the circle of radius $\rho$ centered at $-x^*$. Hence, one can cover a small neighborhood of $-x^*$ and $y_1$ by a ball of diameter $d$.
Every neighborhood of $-x^*$ contains red and blue points; point $y_1$ is green (see Fig. \ref{case1}).
We have a contradiction with assumption ($\star$).

\paragraph{Case 2: $f(x^*) = (a,b)$, $a,b > 0$.} Then both points $x^*, - x^*$ are green. 
We may swap blue and green colors to reduce the situation to the next case with the same $x^*$.

\paragraph{Case 3: $f(x^*) = (a,0)$, $a > 0$.} 
We claim that $a > \rho$.  Assume the contrary, i.e.  $x^* \in \overline{C_{blue}}$ and for every $\eta > 0$ there is a red point $z = z_\eta$ such that $\|x^*-z\| \leq \rho + \eta$. Note that if $x^*$ is green, then it contradicts ($\star$), so $x^*$ is blue. 
There are distinct points $y_1,y_2 \in \overline{B_\rho(-x^*)}$ such that $\|x^*-y_1\| = \|x^*-y_2\| = \|z - y_1\| = \|z - y_2\| = 1$.
Since $x^*$ is blue and $z$ is red $y_1,y_2 \in \overline{C_{green}}$. 
Recall that $f(-x^*) = f(x^*)$, so there is a point $z' \in \overline{C_{red}} \cap \overline{B_{\rho}(-x^*)}$. 
Let $y' \in \{y_1,y_2\}$ be such that $z'$, $-x^*$ and $y'$ do not lie on a great circle of $S^2(r)$. 
Then for a small enough $\eta$ the neighborhoods of $-x^*$, $y'$ and $z'$ can be covered by a ball of diameter $d$.
This is a contradiction with ($\star$).

So the set $\overline{B_\rho(x^*)} \cup \overline{B_\rho(-x^*)}$ is colored with blue and green.

\begin{lemma}
The bipartite subgraph of $S^2(r)$ with parts $\overline{B_\rho(x^*)}$ and $\overline{B_\rho(-x^*)}$ is connected. 
\label{connected}
\end{lemma}

\begin{proof}
Any point $x \in \overline{B_\rho(x^*)}$ has a common neighbor with $x^*$ since the corresponding unit circles intersect.
So $\overline{B_\rho(x^*)}$ belong to the same connected component; the same holds for $\overline{B_\rho(-x^*)}$. 
There is an edge between $\overline{B_\rho(x^*)}$ and $\overline{B_\rho(-x^*)}$, and so the subgraph is connected.
\end{proof}

By Lemma~\ref{connected}, one can color $\overline{B_\rho(x^*)} \cup \overline{B_\rho(-x^*)}$ in two colors in the unique way (up to symmetry): the first part is blue and the second one is green.
Then the distance from $x^*$ and $-x^*$ to $\overline{C_{blue}}$ is zero and nonzero simultaneously.

This contradiction implies that each color is dense in the sphere.

\subsection{Step 2. Stability of embedding}

In this section we will need the implicit function theorem \cite{implicitfunc} in the following weakened formulation.
\begin{theorem}
Let $F:\; \mathbb{R}^{2s} \to \mathbb{R}^s$ be a continuously differentiable function, 
\[
F=F(X,Y) = F(x_1, \dots, x_s; \;y_1, \dots , y_s),
\]
and at some point $X = a$, $Y = b$  the following conditions are satisfied
\[
F(a, b) = 0, \quad \quad \quad
\det \left( \frac{\partial F(X,Y)}{\partial Y } \right)_{X = a, Y = b} \neq 0.
\]
Then there exists $\eta>0$ such that the system of equations $F(X, Y) = 0$ is solvable in $Y$ for any $X$ satisfying the condition $\|X-a\|< \eta$. 
\label{implfunc}
\end{theorem}

Recall that $G_{k}$ is an odd cycle of length $m = 2k+1$ with an extra pendant (leaf) vertex attached to each vertex of the cycle.
In particular, $G_k$ has $2m$ vertices and $2m$ edges. 

Denote by $y_1,\dots, y_m$ the points of $S^2(r)$ that correspond to the cycle vertices and by $x_1,\dots, x_m$ the points of $S^2(r)$ that correspond to the pendant vertices. For convenience, let us put $X = (x_1,\dots, x_m)$ and $Y = (y_1,\dots, y_m)$ the vectors of dimension $s = 3m$ containing all coordinates. Then the embedding of $G_k$ can be given by the pair $(X,Y)$.

\begin{lemma}
Fix the radius $r \in \left(\frac{1}{2} , \frac{\sqrt{3}}{3}\right)$. Then if $k$ is large enough, there exists a unit distance embedding $(X,Y)$ of $G_k$  into $S^2(r)$ and a constant $\eta>0$ such that for any $\tilde X$ satisfying $\|\tilde X-X\|<\eta$ there exists $Y$ such that $(\tilde X, \tilde Y)$ is a ``perturbed'' unit distance embedding of $G_k$.
\end{lemma}

In other words, for any sufficiently small perturbation of pendant vertices, it is possible to find the embedding of the cycle vertices.

\begin{proof}
We provide the desired unit distance embedding explicitly. In what follows we slightly abuse the notation and write $x_i$ and $y_i$ for a vertex of the graph, the corresponding point on $S^2(r)$, and its 3-dimensional vector representation.
Consider the system of equations defining the embedding $G_k$ in $S^2(r)$:

\begin{equation}
\begin{cases}
f_i = \|y_i\|^2 - r^2 = 0, \quad 1\leq i\leq m; \\
f_{i+m} =\|y_i - y_{i+1}\|^2 - 1 = 0, \quad 1\leq i\leq m-1; \\
f_{2m} = \|y_m - y_{1}\|^2 - 1 = 0;\\
f_{i+2m} = \|x_i - y_i\|^2 - 1 = 0, \quad 1 \leq i \leq m.
\end{cases}
\label{embed_sys}
\end{equation}

Next, we will be interested in the family of embeddings, the $k=2$ case of which is depicted on Fig.~\ref{fig:my_label2}.

\begin{figure}[H]
    \centering
    \begin{tikzpicture}[line cap=round,line join=round,>=triangle 45,x=1.0cm,y=1.0cm, scale=0.50]
\draw(8,-6.68) circle (6.92cm);
\draw [rotate around={-1.43:(8.01,-5.85)},dotted] (8.01,-5.85) ellipse (6.84cm and 0.79cm);
\draw [rotate around={-1.29:(7.98,-8.13)},dotted] (7.98,-8.13) ellipse (6.75cm and 0.95cm);
\draw (12.38,-6.57)-- (1.7,-5.39);
\draw (1.7,-5.39)-- (13.42,-5.5);
\draw (13.42,-5.5)-- (4.07,-6.4);
\draw (4.07,-6.4)-- (7.56,-5.05);
\draw (7.56,-5.05)-- (12.38,-6.57);
\draw (13.42,-5.5)-- (5.22,-8.94);
\draw (4.07,-6.4)-- (6.9,-7.17);
\begin{scriptsize}
\draw[color=black] (5.08,-1.32) node {$S^2(r)$};
\draw[color=black] (4.38,-4.86) node {$z=h$};
\draw[color=black] (3.74,-8.46) node {$z=-h$};
\fill [color=black] (12.38,-6.57) circle (1.5pt);
\fill [color=black] (1.7,-5.39) circle (1.5pt);
\fill [color=black] (13.42,-5.5) circle (1.5pt);
\draw[color=black] (13.86,-5.24) node {$y_1$};
\fill [color=black] (4.07,-6.4) circle (1.5pt);
\draw[color=black] (3.76,-6.92) node {$y_2$};
\fill [color=black] (7.56,-5.05) circle (1.5pt);
\draw[color=black] (8,-4.78) node {$y_3$};
\fill [color=black] (5.22,-8.94) circle (1.5pt);
\draw[color=black] (5.66,-9.48) node {$x_1$};
\fill [color=black] (6.9,-7.17) circle (1.5pt);
\draw[color=black] (6.44,-7.64) node {$x_2$};
\end{scriptsize}
\end{tikzpicture}
    \caption{Unit distance embedding of $G_k$, the $k = 2$ case}
    \label{fig:my_label2}
\end{figure}
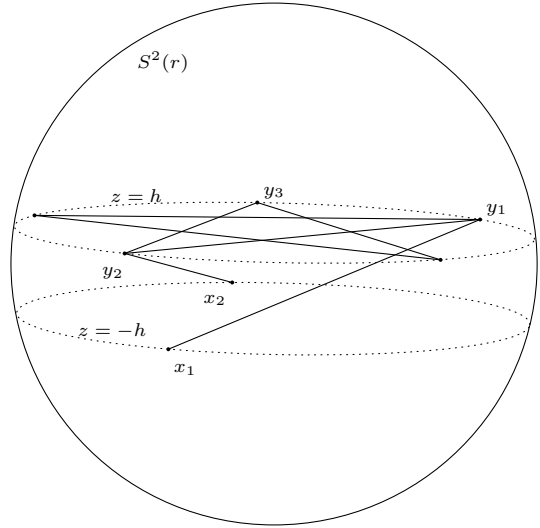

Note that~\eqref{embed_sys} allows $x_i$ to lie in $\mathbb{R}^3$, not only $S^2(r)$, but the cycle $y_1,\dots,y_m$ must lie on the sphere.

One can consider the function corresponding to the left-hand side of the system~\eqref{embed_sys}.
\[
F =  (f_1, \dots , f_{3m}) = F(x_{11}, x_{12}, x_{13},   \dots , x_{m3};\; y_{11}, \dots, y_{m3}) .
\]
Suppose that the Jacobian matrix $J = \left(\frac{\partial F}{ \partial Y}\right)$  is nondegenerate,
\[
\det J = \det \left(\frac{\partial F}{ \partial Y}\right) \neq 0,
\]
then the statement of the lemma follows from Theorem \ref{implfunc}. The rest of the proof is devoted to the calculation of this determinant.

The matrix $J$ has the following form (recall that $x_i$ and $y_i$ are $1\times 3$ vectors):
\begin{equation*}
J(X, Y) =
    2  \begin{pmatrix}
         y_1 & 0 & 0& 0 & \dots & 0\\
        0  &  y_2 & 0 & 0 & \dots & 0\\
        0 & 0 &  y_3 & 0 & \dots & 0\\
        \vdots & &  & \ddots & & \\
        0 & 0 & 0 & 0& \dots &  y_m\\
        y_1 - y_2 & y_2 - y_1 & \dots & 0 & \dots & 0\\
        0 & y_2 - y_3 & y_3 - y_2 & 0 & \dots & 0\\
        \vdots & & & & &\\
        y_1-y_m & 0 & \dots & \dots & 0 & y_{m}-y_1\\
        y_1-x_1 & 0 & \dots & \dots & 0 & 0\\
        0 & y_2-x_2 & 0 & \dots &  0 & 0\\
        \vdots& &\vdots & & \vdots&\\
        
        0& \dots & 0 & \dots & 0 & y_m-x_m\\
    \end{pmatrix}.
\end{equation*}

Subtracting some rows from each other, we get

\begin{multline*}
\det J = 2^{3m} \det
\begin{pmatrix}
         y_1 & 0 & 0& 0 & \dots & 0\\
        0  &  y_2 & 0 & 0 & \dots & 0\\
        0 & 0 &  y_3 & 0 & \dots & 0\\
        \vdots & &  & \ddots & & \\
        0 & 0 & 0 & 0& \dots &  y_m\\
        y_2 & y_1 & \dots & 0 & \dots & 0\\
        0 & y_3 &  y_2 & 0 & \dots & 0\\
        \vdots& & & \ddots & &\\
        y_m & \dots & 0 & \dots & 0 & y_1\\
        x_1 & 0 & 0& 0 & \dots & 0\\
        \vdots & &  & \ddots & & \\
        0 & 0 & 0 & 0& \dots &  x_m\\
    \end{pmatrix}
    =
    2^{3m} \det
    \begin{pmatrix}
         y_1 & 0 & 0& 0 & \dots & 0\\
         x_1 & 0 & 0 & 0 & \dots & 0\\
         y_2 & y_1 & 0 & \dots & \dots & 0\\
        0  &  y_2 & 0 & 0 & \dots & 0\\
        0  &  x_2 & 0 & 0 & \dots & 0\\
        0 &  y_3 &  y_2 & 0 & \dots & 0\\
        \vdots & &  & \ddots & & \\
        0 & 0 & 0 & 0& \dots &  y_m\\
        0 & 0 & 0 & 0& \dots &  x_m\\
        y_m & 0 & 0 & \dots & 0 & y_1\\
    \end{pmatrix} = \\
\end{multline*} 
\[
    = 2^{3m}\left(V_1 \dots V_m + V'_1 \dots V'_m\right),
\]
where
\[
    V_i = - \det \begin{pmatrix}
         y_i \\
         y_{i+1} \\
         x_i
    \end{pmatrix},
\quad\quad
V'_i =  \det \begin{pmatrix}
         y_i \\
         y_{i+1} \\
         x_{i+1}
    \end{pmatrix}.
\]

Now we fix the following embedding (Fig. \ref{fig:my_label2}). Let vertices $y_i$ lie in the plane $z = h$ (and form a regular $m$-gon), and vertices $x_i$ lie in the plane $z = -h$ (and also form a regular $m$-gon). 
Note that the radius of the circumcircle of the $m$-gon is greater than $1/2$, hence
\begin{equation}
    h < \left( \frac{1}{3} -  \frac{1}{4}\right)^{1/2} = \frac{1}{2 \sqrt{3}} < \frac{1}{2}.
    \label{h_est}
\end{equation}

Denote by $U_m$ the rotation matrix by an angle $2\pi/m$ counterclockwise around $z$--axis. Then $y_{i+1}$ = $U_m y_i$, $x_{i+1}$ = $U_m x_i$. Hence, all $V_i$ coincide and all $V'_i$ also coincide; put $V = V_i$ and $V' = V'_i$. Hence
\[
\det J = V^m + (V')^m.
\]

We claim that
\[
V+V' = \det \begin{pmatrix}
     y_1 \\
     y_2 \\
     x_2 - x_1
\end{pmatrix} \neq 0.
\]
Indeed, since $y_{13}=y_{23}=h$, $x_{13}=x_{23}=-h$, the equality
\[ 
\alpha y_1 + \beta y_2 + \gamma (x_2-x_1) = 0 
\]
implies $\alpha = - \beta$, i.e.
\begin{equation}
 \alpha(y_1-y_2) = \gamma(x_1 - x_2).
 \label{collinear}
\end{equation}
Recall that $\|y_1-y_2\| = \|x_1 - x_2\|=1$, so $\alpha=\pm\gamma$. 

Since both sets of points $\mathcal{X} = \{x_1, \dots , x_m\}$, $\mathcal{Y} = \{y_1, \dots , y_m\}$ form vertices of congruent regular $m$-gons, in the case  $\alpha=\gamma$, we have $x_1-x_2=y_1-y_2$ and the projections of $x_i$ and $y_i$ on the plane $z=0$ coincide, $i = 1,2, \dots, m$, and taking into account~\eqref{h_est}, we have
\[
\|x_1-y_1\| = 2h < 1.
\]

In the case $\alpha=-\gamma$, we have $x_1-x_2=y_2-y_1$ and the sets $\mathcal X$ and $\mathcal Y$ are symmetric about the origin. 
Then $x_1 x_2 y_1 y_2$ is a rectangle, and
\[
\|x_1-y_1\|^2 > \|x_1-x_2\|^2 + 4h^2 > 1.
\]
In both cases we got a contradiction.
Then the equation~\eqref{collinear} does not hold and so $V + V' \neq 0$. Hence 
\[
  \det J = V^m+(V')^m \neq 0
\]
as required.
\end{proof}

\section{Open questions}

\paragraph{Does the chromatic number of $S^2(r)$ <<almost>> grow with $r$?} 
Id est is the chromatic number monotonic except for at most countable set of values $r$?
Recall that the known results (see Table 1) allow for such possibility.

\begin{table}[H]
    \centering
    \begin{tabular}{|c|c|c|}
    \hline
        $r$ & Estimate for $\chi(r)=\chi(S^2(r))$ & Source  \\
        \hline
        $r<1/2$ & $\chi(r)=1$ &   \\
         \hline
        $r = 1/2$ & $\chi(r)=2$ & \\
         \hline
        $\frac{1}{2} < r\leq\frac{\sqrt{3-\sqrt{3}}}{2}$ & $\chi(r) = 4$ & Corollary 1\\
         \hline
        $r > \frac{\sqrt{3-\sqrt{3}}}{2}$ & $\chi(r)\geq 4$ & Theorem 2 \\
        \hline
        $r = \frac{\sqrt{5-\sqrt{5}}}{2\sqrt{2}}$  & $\chi(r) \geq 5$ & \cite{sphere5chr} \\
         \hline
        $r = \frac{1}{\sqrt 2}$ & $\chi(r)=4$ & \cite{simmons1976chromatic,Godsil} \\
        \hline        
        
        $r = \frac{\sqrt{5+\sqrt{5}}}{2\sqrt{2}}$  & $\chi(r) \geq 5$ & \cite{sphere5chr} \\
         \hline
        $r \leq \frac{1}{\sqrt{3}}$ & $\chi(r) \leq 5$ & \cite{simmons1976chromatic, Malen} \\
         \hline
        $r \leq \sqrt{3}/2$ & $\chi(r) \leq 6$ & \cite{Malen} \\
         \hline
        $r \geq 12.44$ & $\chi(r) \leq 7$ & \cite{Sirgedas} \\
         \hline
        $r > 1/2$ & $\chi(r)\leq 15$ & \cite{Coulson,RT} \\
         \hline
    \end{tabular}
    \caption{Lower and upper estimates for $\chi(S^2(r))$.}
    \label{sph_estimates}
\end{table}

\paragraph{Is there a proper coloring of $S^2(r)$ in $\chi(S^2(r))$ colors such that every color is dense?}
It is interesting that all known upper bounds are given by explicit colorings in which every color is a finite union of regions bounded by piecewise-continuous curves.

\paragraph{What is the minimal number of vertices in a subgraph $G$ of a sphere $S^2(r)$ with $\chi(G) = \chi (S^2(r))$?} 
By the de~Bruijn--Erd{\H o}s theorem this number is finite. Note that the proof of Theorem 2 does not give any finite 4-chromatic unit distance graph.

\begin{figure}[H]
    \centering
    \begin{tikzpicture}[line cap=round,line join=round,>=triangle 45,x=1cm,y=1cm, scale=0.5]
\draw [line width=1pt] (-1.44,-0.69) circle (5cm);
\draw [rotate around={-0.252402682289151:(-1.41,3.72)},line width=1pt] (-1.41,3.72) ellipse (2.2841542200439937cm and 0.25369371483087433cm);
\draw [shift={(0.8570946669734408,-0.4709200597124944)},line width=1pt]  plot[domain=2.3560716282226015:4.294548452653778,variable=\t]({1*5.710312926524084*cos(\t r)+0*5.710312926524084*sin(\t r)},{0*5.710312926524084*cos(\t r)+1*5.710312926524084*sin(\t r)});
\draw [shift={(-4.98251751172431,-0.19657932988760576)},line width=1pt]  plot[domain=-1.0006217325812896:0.6118818044703762,variable=\t]({1*6.525702456942265*cos(\t r)+0*6.525702456942265*sin(\t r)},{0*6.525702456942265*cos(\t r)+1*6.525702456942265*sin(\t r)});
\draw [shift={(-125.255550907007,-2.692768952421932)},line width=1pt,dash pattern=on 5pt off 5pt]  plot[domain=-0.024206098160666656:0.05386033886944003,variable=\t]({1*123.8317476654611*cos(\t r)+0*123.8317476654611*sin(\t r)},{0*123.8317476654611*cos(\t r)+1*123.8317476654611*sin(\t r)});
\draw (-2.5905418037294092,-0.8144522400481754) node[anchor=north west] {$2$};
\draw (-5.788493930220483,-0.1944411134835805) node[anchor=north west] {$1$};
\draw (2.638762566374616,-0.0068061672864004755) node[anchor=north west] {$3$};
\draw (-1.4647321265463276,5.3) node[anchor=north west] {$0$};
\begin{scriptsize}
\draw [fill=black] (-3.1802102034113178,3.567376995036643) circle (2.5pt);
\draw [fill=black] (-1.4600801593516382,-5.689959678557459) circle (2.5pt);
\draw [fill=black] (0.3592172383589387,3.551841993322389) circle (2.5pt);
\draw [fill=black] (-1.603373818776061,3.9736267171436577) circle (2.5pt);
\end{scriptsize}
\end{tikzpicture}
    \caption{4-coloring of the sphere. Here $s_0 \to 0$ as $r \to 1/2$}
    \label{fig:my_label3}
\end{figure}
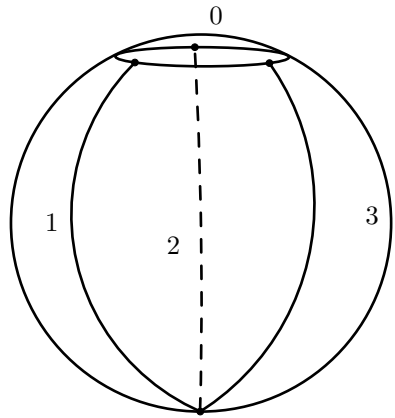

Let us focus on the case $r = 1/2 +\varepsilon$,  $\varepsilon \to 0$. Then the sphere can be colored in 4 colors in the way shown in Figure~\ref{fig:my_label3}.  Let us denote by $s_0$ the area of the spherical cap of color 0. Observe that $s_0 = 4\pi\varepsilon+o(\varepsilon)$, and thus, via averaging, we have the lower bound $n_4(r) \geq c\varepsilon^{-1}$ for some $c>0$, where $n_4(r)$ is the minimal number of vertices in a 4-chromatic unit distance graph. Can this obvious bound be refined?

\paragraph{Acknowledgements.} The research is supported by <<Native towns>>, a social investment program of PJSC <<Gazprom Neft>>, and by the program <<Leading Scientific Schools>> under grant NSh-775.2022.1.1.
We are grateful to Alexei Gordeev for helping to write the manuscript and to D{\"o}m{\"o}t{\"o}r P{\'a}lv{\"o}lgyi for comments that helped us improve the readability of the text. Finally, Andrey Kupavskii has significantly refined the explanation.

\bibliographystyle{plain}
\bibliography{main}

\end{document}